\documentclass{article}%
\usepackage{hyperref}
\usepackage{amsmath}
\usepackage{amsfonts}
\usepackage{amssymb}
\usepackage{graphicx}%
\providecommand{\U}[1]{\protect\rule{.1in}{.1in}}
\newtheorem{theorem}{Theorem}

\newtheorem{lemma}{Lemma}[section]

\newtheorem{remark}{Remark}

\newenvironment{proof}[1][\sl Proof]{\noindent\textbf{#1.} }{\
$\square$
}
\numberwithin{equation}{section}
\begin{document}

\title{\textbf{Stability of peakons for the Degasperis-Procesi equation}}
\author{Zhiwu Lin\thanks{Department of Mathematics, University of Missouri, Columbia,
MO 65211, lin@math.missouri.edu} \ and Yue Liu\thanks{Department of
Mathematics, University of Texas, Arlington, TX 76019, yliu@uta.edu}}
\date{}
\maketitle

\begin{abstract}
\noindent The Degasperis-Procesi equation can be derived as a member of a
one-parameter family of asymptotic shallow water approximations to the Euler
equations with the same asymptotic accuracy as that of the Camassa-Holm
equation.  In
this paper, we study the orbital stability problem of the peaked solitons to
the Degasperis-Procesi equation on the line. By constructing a Liapunov
function, we prove that the shapes of these peakon solitons are stable under
small perturbations. \newline


\vskip 0.2cm


\end{abstract}

\section{Introduction}

The Degasperis-Procesi (DP) equation
\begin{equation}
y_{t}+y_{x}u+3yu_{x}=0,\ \ \ \ x\in\mathbb{R},\ \ t>0, \label{DP}%
\end{equation}
with $y=u-u_{xx},$ was originally derived by Degasperis-Procesi \cite{D-P}
using the method of asymptotic integrability up to third order as one of three
equations in the family of third order dispersive PDE conservation laws of the
form
\begin{equation}
u_{t}-\alpha^{2}u_{xxt}+\gamma u_{xxx}+c_{0}u_{x}=(c_{1}u^{2}+c_{2}u_{x}%
^{2}+c_{3}uu_{xx})_{x}. \label{1.1}%
\end{equation}
The other two integrable equations in the family, after rescaling and applying
a Galilean transformation, are the Korteweg-de Vries (KdV) equation
\[
u_{t}+u_{xxx}+uu_{x}=0
\]
and the Camassa-Holm (CH) shallow water equation \cite{C-H, D-G-H1, F-F, J},
\begin{equation}
y_{t}+y_{x}u+2yu_{x}=0,\ \ \ y=u-u_{xx}. \label{ch}%
\end{equation}
These three cases exhaust in the completely integrable candidates for
(\ref{1.1}) by Painlev\'{e} analysis. Degasperis, Holm and Hone \cite{D-H-H}
showed the formal integrability of the DP equation as Hamiltonian systems by
constructing a Lax pair and a bi-Hamiltonian structure.

The Camassa-Holm equation was first derived by Fokas and Fuchassteiner
\cite{F-F} as a bi-Hamiltonian system, and then as a model for shallow water
waves by Camassa and Holm \cite{C-H}. The DP equation is also in dimensionless
space-time variables $(x,t)$ an approximation to the incompressible Euler
equations for shallow water under the Kodama transformation \cite{H-S, H-S2}
and its asymptotic accuracy is the same as that of the Camassa-Holm (CH)
shallow water equation, where $u(t,x)$ is considered as the fluid velocity at
time $t$ in the spatial $x$-direction with momentum density $y.$

Recently, Liu and Yin \cite{L-Y} proved that the first blow-up in finite time
to equation (\ref{DP}) must occur as wave breaking and shock waves possibly
appear afterwards. It is shown in \cite{L-Y} that the lifespan of solutions of
the DP equation (\ref{DP}) is not affected by the smoothness and size of the
initial profiles, but affected by the shape of the initial profiles (for the CH equation, see \cite {B-C, C-E}). This can
be viewed as a significant difference between the DP equation (or the CH
equation ) and the KdV. It is also noted that the KdV equation, unlike the CH
equation or DP equation, does not have wave breaking phenomena \cite{Ta}. Under wave breaking we understand that development of singularities in finite time by which the wave remains bounded but its slope becomes unbounded \cite{Wh}.

It is well known that the KdV equation is an integrable Hamiltonian equation
that possesses smooth solitons as traveling waves. In the KdV equation, the
leading order asymptotic balance that confines the traveling wave solitons
occurs between nonlinear steepening and linear dispersion. However, the
nonlinear dispersion and nonlocal balance in the CH equation and the DP
equation, even in the absence of linear dispersion, can still produce a
confined solitary traveling waves
\begin{equation}
u(t,x)=c\varphi(x-ct) \label{peakon}%
\end{equation}
traveling at constant speed $c>0,$ where $\varphi(x)=e^{-|x|}.$ Because of
their shape (they are smooth except for a peak at their crest), these
solutions are called the peakons \cite{C-H, D-H-H}. Peakons of both equations
are true solitons that interact via elastic collisions under the CH dynamics,
or the DP dynamics, respectively. The peakons of the CH equation are orbitally
stable \cite{C-S}. For waves that approximate the peakons in a special way, a
stability result was proved by a variation method \cite{C-M2}.

Note that we can rewrite the DP equation as
\begin{equation}
u_{t}-u_{txx}+4uu_{x}=3u_{x}u_{xx}+uu_{xxx},\quad t>0,\;x\in\mathbb{R}.
\label{1.3}%
\end{equation}

The peaked solitons are not classical solutions of (\ref{1.3}). They satisfy
the Degasperis-Procesi equation in the conservation law form
\begin{equation}
u_{t}+\partial_{x}\left(  \frac{1}{2}u^{2}+\frac{1}{2}\varphi\ast\left(
\frac{3}{2}u^{2}\right)  \right)  =0,\ \ t>0,\ \ x\in\mathbb{R},
\label{DP-conser}%
\end{equation}
where $\ast$ stands for convolution with respect to the spatial variable
$x\in\mathbb{R}.$ This is the exact meaning in which the peakons are solutions.

Recently, Lundmark and Szmigielski \cite{L-S} presented an inverse scattering
approach for computing n-peakon solutions to equation (\ref{1.3}). Holm and
Staley \cite{H-S} studied stability of solitons and peakons numerically to
equation (\ref{1.3}). Analogous to the case of Camassa-Holm equation
\cite{Cf2}, Henry \cite{H} and Mustafa \cite{Mu} showed that smooth solutions to equation
(\ref{1.3}) have infinite speed of propagation.

The following are three useful conservation laws of the Degasperis-Procesi
equation.
\[
E_{1}(u)=\int_{\mathbb{R}}y\,dx,\;\ \ E_{2}(u)=\int_{\mathbb{R}}%
yv\,dx,\;\ \ E_{3}(u)=\int_{\mathbb{R}}u^{3}\,dx,
\]
where $y=(1-\partial_{x}^{2})u$ and $v=(4-\partial_{x}^{2})^{-1}u$, while the
corresponding three useful conservation laws of the Camassa-Holm equation are
the following:
\begin{equation}
F_{1}(u)=\int_{\mathbb{R}}y\,dx,\ \;F_{2}(u)=\int_{\mathbb{R}}(u^{2}+u_{x}%
^{2})\,dx,\ \;F_{3}(u)=\int_{\mathbb{R}}(u^{3}+uu_{x}^{2})\,dx.
\label{ch-invariants}%
\end{equation}

The stability of solitary waves is one of the fundamental qualitative
properties of the solutions of nonlinear wave equations. Numerical simulations
\cite{D-H-H, Lu} suggest that the sizes and velocities of the peakons do not
change as a result of collision so these patterns are expected to be stable.
Furthermore, it is observed that the shape of the peakons remains
approximately the same as time evolves. As far as we know, the case of
stability of the peakons for the Camassa-Holm equation is well understood by
now \cite{C-M2, C-S}, while the Degasperis-Procesi equation case is the
subject of this paper. The goal of this paper is to establish a stability
result of peaked solitons for equation (\ref{1.3}).

It is found that the corresponding conservation laws of the Degasperis-Procesi
equation are much weaker than those of the Camassa-Holm equation. In
particular, one can see that the conservation law $E_{2}(u)$ for the DP
equation is equivalent to $\Vert u\Vert_{L^{2}}^{2}.$ In fact, by the Fourier
transform, we have
\begin{equation}
E_{2}(u)=\int_{\mathbb{R}}yvdx=\int_{\mathbb{R}}\frac{1+\xi^{2}}{4+\xi^{2}%
}|\hat{u}(\xi)|^{2}d\xi\sim\Vert\hat{u}\Vert_{L^{2}}^{2}=\Vert u\Vert_{L^{2}%
}^{2}. \label{energy-DP}%
\end{equation}
Therefore, the stability issue of the peaked solitons of the DP equation is
more subtle .

For the DP equation, we can only expect to obtain the orbital stability of
peakons in the sense of $L^{2}-$norm due to a weaker conservation law $E_{2}.$
The solutions of the DP equation usually tend to be oscillations which spread
out spatially in a quite complicated way. In general, a small perturbation of
a solitary wave can yield another one with a different speed and phase shift.
We define the orbit of traveling-wave solutions $c\varphi$ to be the set
$U(\varphi)=\{c\varphi(\cdot+x_{0}),\ x_{0}\in\mathbb{R}\},$ and a peaked
soliton of the DP equation is called orbitally stable if a wave starting close
to the peakon remains close to some translate of it at all later times.

Let us denote
\[
E_{2}(u)=\Vert u\Vert_{X}^{2}.
\]

The following stability theorem is the principal result of the present paper.

\begin{theorem}
[Stability]\label{th1} Let $c\varphi$ be the peaked soliton defined in
(\ref{peakon}) traveling with speed $c>0.$ Then $c\varphi$ is orbitally stable
in the following sense. If $u_{0}\in H^{s}$ for some $s>3/2,$ $y_{0}%
=u_{0}-\partial_{x}^{2}u_{0}\ $is a nonnegative Radon measure of finite total
mass, and
\[
\Vert u_{0}-c\varphi\Vert_{X}<c\varepsilon,\ \ |E_{3}(u_{0})-E_{3}%
(c\varphi)|<c^{3}\varepsilon,\ \text{ }0<\varepsilon<\frac{1}{2},
\]
then the corresponding solution $u(t)$ of equation (\ref{1.3}) with initial
value $u(0)=u_{0}$ satisfies
\[
\sup_{t\geq0}\Vert u(t,\cdot)-c\varphi(\cdot-\xi_{1}(t))\Vert_{X}%
<3c\ \varepsilon^{1/4},
\]
where $\xi_{1}(t)\in\mathbb{R}$ is the maximum point of the function
$v(t,\cdot)=(4-\partial_{x}^{2})^{-1}u(t,\cdot)$. Moreover, let
\[
M_{1}\left(  t\right)  =v(t,\xi_{1}(t))\geq M_{2}\left(  t\right)  \cdots\geq
M_{n}\left(  t\right)  \geq0\text{\ and}\ m_{1}\left(  t\right)  \geq
\cdots\geq m_{n-1}\left(  t\right)  \geq0
\]
be all local maxima and minima of the nonnegative function $v(t,\cdot),$
respectively. Then
\begin{equation}
\left\vert M_{1}\left(  t\right)  -\frac{c}{6}\right\vert \leq c\sqrt
{2\varepsilon}\label{ineq-M1-thm}%
\end{equation}
and
\begin{equation}
\sum_{i=2}^{n}\left(  M_{i}^{2}\left(  t\right)  -m_{i-1}^{2}\left(  t\right)
\right)  <2c^{2}\sqrt{\varepsilon}.\label{ineq-M2-thm}%
\end{equation}

\end{theorem}

\begin{remark} The state of affairs about these maxima/minima implied by the previous theorem is a consequence of the assumption on $ y_0, $ as shown in Lemma \ref{le3.1}.
For an initial profile $u_{0}\in H^{s},\ s>3/2,$ there exists a local solution
$u\in C([0,T),H^{s})\ $of (\ref{1.3}) with initial data $u(0)=u_{0}$
\cite{Y1}. Under the assumption $y_{0}=u_{0}-\partial_{x}^{2}u_{0}\geq0\ $in
Theorem \ref{th1}, the existence is global in time \cite{L-Y}, that is
$T=+\infty$. For peakons $c\varphi$ with $c>0$, we have $\left(
1-\partial_{x}^{2}\right)  \left(  c\varphi\right)  =2c\delta$ (here $\delta$
is the Dirac distribution). Hence the assumption on $y_{0}$ that it is a
nonnegative measure is quite natural for a small perturbation of the peakons.
Existence of global weak solution in $H^{1}$ of the DP equation is also proved
in \cite{E-L-Y}. Note that peakons $c\varphi$ are not strong solutions, since
$\varphi\in H^{s},$ only for $s<3/2.$

The above theorem of orbital stability states that any solution starting close
to peakons $c\varphi$ remains close to some translate of $c\varphi$ in the
norm $\Vert\ \Vert_{X}$, at any later time. More information about this
stability is contained in (\ref{ineq-M1-thm}) and (\ref{ineq-M2-thm}). Notice
that for peakons $c\varphi$, the function $v_{c\varphi}$ is single-humped with
the height $\frac{1}{6}c$. So (\ref{ineq-M1-thm}) and (\ref{ineq-M2-thm})
imply that the graph of $v(t,\cdot)$ is close to that of the peakon $c\varphi$
with a fixed $c>0$ for all times.
\end{remark}

There are two standard methods to study stability issues of dispersive wave
equations. One is the variational approach which constructs the solitary waves
as energy minimizers under appropriate constraints, and the stability
automatically follows. However, without uniqueness of the minimizer, one can
only obtain the stability of the set of minima. The variational approach is
used in \cite{C-M2} for the CH equation. It is shown in \cite{C-M2} that the
each peakon $c\varphi$ is the unique minimum (\textsl{ground state}) of
constrained energy, from which its orbital stability is proved for initial
data $u_{0}\in H^{3}$ with $y_{0}=(1-\partial_{x}^{2})u_{0}\geq0$. Their proof
strongly relies on the fact that the conserved energy $F_{2}$ in
(\ref{ch-invariants}) of the CH equation is the $H^{1}-$norm of the solution.
However, for the DP equation the energy $E_{2}$ in (\ref{energy-DP}) is only
the $L^{2}$ norm of the solution. Consequently, it is more difficult to use
such a variational approach for the DP equation.

Another approach to study stability is to linearize the equation around the
solitary waves, and it is commonly believed that nonlinear stability is
governed by the linearized equation. However, for the CH and DP equations, the
nonlinearity plays the dominant role rather than being a higher-order
correction to linear terms. Thus it is unclear how one can get nonlinear
stability of peakons by studying the linearized problem. Morover, the peaked
solitons $c\varphi$ are not differentiable, which makes it difficult to
analyze the spectrum of the linearized operator around $c\varphi.$

To establish the stability result for the DP equation, we extend the approach
in \cite{C-S} for the CH equation. The idea in \cite{C-S} is to directly use
the energy $F_{2}$ as the Liapunov functional. By expanding $F_{2}$ in
(\ref{ch-invariants}) around the peakon $c \varphi$, the error term is in the
form of the difference of the maxima of $c \varphi$ and the perturbed solution
$u$. To estimate this difference, they establish two integral relations
\[
\int g^{2}=F_{2}\left(  u\right)  -2\left(  \max u\right)  ^{2}
\ \ \mathrm{and} \ \ \ \int ug^{2}=F_{3}\left(  u\right)  -\frac{4}{3}\left(
\max u\right)  ^{3}
\]
with a function $g. $ Relating these two integrals, one can get%
\[
F_{3}(u)\leq MF_{2}(u)-\frac{2}{3}M^{3},\ \ M=\max u(x)
\]
and the error estimate $\left\vert M-\max\varphi\right\vert $ then follows
from the structure of the above polynomial inequality.

To extend the above approach to nonlinear stability of the DP peakons, we have
to overcome several difficulties. By expanding the energy $E_{2}\left(
u\right)  $ around the peakon $c\varphi$, the error term turns out to be $\max
v_{c\varphi}-\max v_{u}$, with $v_{u}=(4-\partial_{x}^{2})^{-1}u$. We can
derive the following two integral relations for $M_{1}=\max v_{u},$
$E_{2}\left(  u\right)  $ and $E_{3}(u)$ by
\[
\int g^{2}=E_{2}\left(  u\right)  -12M_{1}^{2}\ \ \ \mathrm{and}\ \ \ \int
hg^{2}=E_{3}\left(  u\right)  -144M_{1}^{3}%
\]
with some functions $g$ and $h$ related to $v_{u}.$ To get the required
polynomial inequality from the above two identities, we need to show
$h\leq18\max v_{u}$. However, since $h$ is of the form $-\partial_{x}^{2}%
v_{u}\pm6\partial_{x}v_{u}+16v_{u},$ generally it can not be bounded by
$v_{u}$. This new difficulty is due to the more complicated nonlinear
structure and weaker conservation laws of the DP equation. To overcome it, we
introduce a new idea. By constructing $g$ and $h$ piecewise according to
monotonicity of the function $v_{u},$ we then establish two new integral
identities (\ref{3.7}) and (\ref{3.12}) for $E_{2},E_{3}$ and all local maxima
and minima of $v_{u}$. The crucial estimate $h\leq18\max v_{u}$ can now be
shown by using this monotonicity structure and properties of the DP solutions.
This results in inequality (\ref{3.18}) related to $E_{2},E_{3}$ and all local
maxima and minima of $v_{u}$. By analyzing the structure of equality
(\ref{3.18}), we can obtain not only the error estimate $\left\vert M_{1}-\max
v_{c\varphi}\right\vert $ but more precise stability information from
(\ref{ineq-M2-thm}). We note that the same approach can also be used for the
CH equation to gain more stability information (see Remark \ref{remark-proof}).

Although the DP equation is similar to the CH equation in several aspects, we
would like to point out that these two equations are truly different. One of
the novel features of the DP equation is it has not only peaked solitons
\cite{D-H-H}, $u(t,x)=ce^{-|x-ct|},\,c>0$ but also shock peakons \cite{C-K-R,
Lu} of the form
\begin{equation}
u(t,x)=-\frac{1}{t+k}\text{sgn}(x)e^{-|x|},\,k>0. \label{1.5}%
\end{equation}
It is noted that the above shock-peakon solutions \cite{Lu} can be observed by
substituting $(x,t)\longmapsto(\epsilon x,\epsilon t)$ to equation (\ref{1.3})
and letting $\epsilon\rightarrow0$ so that it yields the \textquotedblleft
derivative Burgers equation\textquotedblright\ $\displaystyle\left(
u_{t}+uu_{x}\right)  _{xx}=0,$ from which shock waves form. The periodic shock
waves were established by Escher, Liu and Yin \cite{E-L-Y2}.

The shock peakons can be also observed from the collision of the peakons
(moving to the right) and antipeakons (moving to left) \cite{Lu}.

For example, if we choose the initial data
\[
u_{0}(x)=c_{1}e^{-\mid x-x_{1}\mid} - c_{1}e^{-\mid x-x_{2}\mid},
\]
with $c_{1} > 0 $, and $x_{1} + x_{2}= 0, x_{2} > 0 $, then the collision
occurs at $x = 0 $ and the solution
\[
u(x,t)= p_{1}(t)e^{-\mid x-q_{1}(t)\mid}+p_{2}(t)e^{-\mid x-q_{2}(t)\mid},
\]
$(x,t)\in\mathbb{R}_{+}\times\mathbb{R}, $ only satisfies the DP equation for
$t < T. $ The unique continuation of $u(x, t) $ into an entropy weak solution
is then given by the stationary decaying shock peakon
\[
u(x, t) = \frac{ - sgn(x) e^{-|x|}} {k + (t - T)} \quad\mathrm{for } \; t \geq
T.
\]

On the other hand, the isospectral problem in the Lax pair for equation
(\ref{1.3}) is the third-order equation
\[
\psi_{x}-\psi_{xxx}-\lambda y\psi=0
\]
cf. \cite{D-H-H}, while the isospectral problem for the Camassa-Holm equation
is the second order equation
\[
\psi_{xx}-\frac{1}{4}\psi-\lambda y\psi=0
\]
(in both cases $y=u-u_{xx}$) cf. \cite{C-H}. Another indication of the fact
that there is no simple transformation of equation (\ref{1.3}) into the
Camassa-Holm equation is the entirely different form of conservation laws for
these two equations \cite{C-H, D-H-H}. Furthermore, the Camassa-Holm equation
is a re-expression of geodesic flow on the diffeomorphism group \cite{Cf, C-K} and 
on  the Bott-Virasoro group \cite{C-K-K-T, Mi}, while no such geometric derivation of
the Degasperis-Procesi equation is available.

The remainder of the paper is organized as follows. In Section 2, we recall
the local well-posedness of the Cauchy problem of equation (\ref{1.3}), the
precise blow-up scenario of strong solutions, and several useful results which
are crucial in the proof of stability theorem for equation (\ref{1.3}) from
\cite{Y1,Y4}. Section 3 is devoted to the proof of the stability result
(Theorem \ref{th1}). \bigskip

\noindent\textit{Notation.} As above and henceforth, we denote by $\ast$
convolution with respect to the spatial variable $x\in\mathbb{R}.$ We use
$\Vert\cdot\Vert_{L^{p}}$ to denote the norm in the Lebesgue space
$L^{p}(\mathbb{R})$ $(1\leq p\leq\infty),$ and $\Vert\cdot\Vert_{H^{s}%
},\,s\geq0$ for the norm in the Sobolev spaces $H^{s}(\mathbb{R}).$

\section{Preliminaries}

In the present section, we discuss the issue of well-posedness. The local
existence theory of the initial-value problem is necessary for our study of
nonlinear stability. We briefly collect the needed results from \cite{L-Y, Y1,
Y4}.

Denote $p(x):=\frac{1}{2}e^{-|x|}$, $x\in\mathbb{R}$, then $(1-\partial
_{x}^{2})^{-1}f=p\ast f$ for all $f\in L^{2}(\mathbb{R})$ and $p\ast
(u-u_{xx})=u$. Using this identity, we can rewrite the DP equation (\ref{1.3})
as follows:
\begin{equation}
u_{t}+uu_{x}+\partial_{x}p\ast\left(  \frac{3}{2}u^{2}\right)  =0,\quad
t>0,\;x\in\mathbb{R}. \label{2.1}%
\end{equation}

The local well-posedness of the Cauchy problem of equation (\ref{1.3}) with
initial data $u_{0}\in H^{s}(\mathbb{R}),\,s>\frac{3}{2}$ can be obtained by
applying Kato's theorem \cite{K, Y1}. As a result, we have the following
well-posedness result.

\begin{lemma}
\cite{Y1} Given $u_{0}\in H^{s}(\mathbb{R}),\;s>\frac{3}{2}$, there exist a
maximal $T=T(u_{0})>0$ and a unique solution $u$ to equation (\ref{1.3}) (or
equation (\ref{2.1})), such that
\[
u=u(\cdot,u_{0})\in C([0,T);H^{s}(\mathbb{R}))\cap C^{1}([0,T);H^{s-1}%
(\mathbb{R})).
\]
Moreover, the solution depends continuously on the initial data, i.e. the
mapping $u_{0}\mapsto u(\cdot,u_{0}):H^{s}(\mathbb{R})\rightarrow
C([0,T);H^{s}(\mathbb{R}))\cap C^{1}([0,T);H^{s-1}(\mathbb{R}))$ is continuous
and the maximal time of existence $T>0$ can be chosen to be independent of
$s.$
\end{lemma}

The following two lemmas show that the only way that a classical solution to
(\ref{1.3}) may fail to exist for all time is that the wave may break.

\begin{lemma}
\cite{Y1} Given $u_{0} \in H^{s}(\mathbb{R}),\,s > \frac{3}{2}$, blow up of
the solution $u=u(\cdot,u_{0})$ in finite time T $< + \infty$ occurs if and
only if
\[
\liminf_{t \uparrow T} \{\inf_{x \in\mathbb{R}}[u_{x} (t,x)] \} = - \infty.
\]

\end{lemma}

\begin{lemma}
\cite{L-Y} Assume $u_{0}\in H^{s}(\mathbb{R}),\,s>\frac{3}{2}$. Let $T$ be the
maximal existence time of the solution $u$ to equation (\ref{1.3}). Then we
have
\[
\Vert u(t,x)\Vert_{L^{\infty}}\leq3\Vert u_{0}(x)\Vert_{L^{2}}^{2}t+\Vert
u_{0}(x)\Vert_{L^{\infty}},\quad\forall t\in\lbrack0,T].
\]

\end{lemma}

Now consider the following differential equation
\begin{equation}
\left\{
\begin{array}
[c]{ll}%
q_{t}=u(t,q),\quad & t\in\lbrack0,T),\\
q(0,x)=x, & x\in\mathbb{R}.
\end{array}
\right.  \label{2.2}%
\end{equation}

Applying classical results in the theory of ordinary differential equations,
one can obtain the following two results on $q $ which are crucial in the
proof of global existence and blow-up solutions.

\begin{lemma}
\cite{Y4}\label{le2.1} Let $u_{0}\in H^{s}(\mathbb{R}),\;s\geq3$, and let
$T>0$ be the maximal existence time of the corresponding solution $u$ to
equation(\ref{1.3}). Then the equation (\ref{2.2}) has a unique solution $q\in
C^{1}([0,T)\times\mathbb{R},\mathbb{R})$. Moreover, the map $q(t,\cdot)$ is an
increasing diffeomorphism of $\mathbb{R}$ with
\[
q_{x}(t,x)=\exp\left(  \int_{0}^{t}u_{x}(s,q(s,x))ds\right)  >0,\,\;\forall
(t,x)\in\lbrack0,T)\times\mathbb{R}.
\]

\end{lemma}

\begin{lemma}
\label{le2.5} \cite{Y4} Let $u_{0}\in H^{s}(\mathbb{R}),\;s\geq3$, and let
$T>0$ be the maximal existence time of the corresponding solution $u$ to
equation (\ref{2.2}). Setting $y:=u-u_{xx}$, we have
\[
y(t,q(t,x))q_{x}^{3}(t,x)=y_{0}(x),\quad\forall(t,x)\in\lbrack0,T)\times
\mathbb{R}.
\]

\end{lemma}

The next two lemmas clearly show that the solution of equation (\ref{1.3}) is
affected by the shape of the initial profiles, not the smoothness and size of
the initial profiles.

\begin{lemma}
\cite{L-Y} Let $u_{0}\in H^{s}(\mathbb{R}),s>\frac{3}{2}.$ Assume there exists
$x_{0}\in\mathbb{R}$ such that
\[
\left\{
\begin{array}
[c]{ll}%
y_{0}(x)=u_{0}(x)-u_{0,xx}(x)\geq0\qquad & \text{ if }\quad x\leq x_{0},\\
y_{0}(x)=u_{0}(x)-u_{0,xx}(x)\leq0\qquad & \text{ if }\quad x\geq x_{0},
\end{array}
\right.
\]
and $y_{0}$ changes sign. Then, the corresponding solution to
equation(\ref{1.3}) blows up in a finite time.
\end{lemma}

\begin{lemma}
\label{le2.7} \cite{L-Y} Assume $u_{0}\in H^{s}(\mathbb{R}),\;s>\frac{3}{2}$
and there exists $x_{0}\in\mathbb{R}$ such that
\[
\left\{
\begin{array}
[c]{ll}%
y_{0}(x)\leq0\qquad & \text{ if }\quad x\leq x_{0},\\
y_{0}(x)\geq0\qquad & \text{ if }\quad x\geq x_{0}.
\end{array}
\right.
\]
Then equation (\ref{1.3}) has a unique global strong solution
\[
u=u(.,u_{0})\in C([0,\infty);H^{s}(\mathbb{R}))\cap C^{1}([0,\infty
);H^{s-1}(\mathbb{R})).
\]
Moreover, $E_{2}(u)=\int_{\mathbb{R}}yv\,dx$ is a conservation law, where
$y=(1-\partial_{x}^{2})u$ and $v=(4-\partial_{x}^{2})^{-1}u$, and for all
$t\in\mathbb{R}_{+}$ we have\newline(i) $u_{x}(t,\cdot)\geq-|u(t,\cdot)|$ on
$\mathbb{R}$,\newline(ii) $\Vert u\Vert_{1}^{2}\ \leq6\Vert u_{0}\Vert_{L^{2}%
}^{4}t^{2}+4\Vert u_{0}\Vert_{L^{2}}^{2}\Vert u_{0}\Vert_{L^{\infty}}t+\Vert
u_{0}\Vert_{1}^{2}.$
\end{lemma}

The following lemma is a special case of Lemma \ref{le2.7}.

\begin{lemma}
\label{le2.8} \cite{L-Y} Assume $u_{0}\in H^{s}(\mathbb{R}),\;s>\frac{3}{2}$.
If $y_{0}=u_{0}-u_{0,xx} \ge0 \ (\le0) $ on $\mathbb{R}$, then equation
(\ref{1.3}) has a unique global strong solution $u$ such that $u(t,x) \ge0
\ (\le0) $ and $y(t,x)=u-\partial_{x}^{2}u \ge0 \ (\le0) $ for all
$(t,x)\in\mathbb{R_{+}}\times\mathbb{R}.$
\end{lemma}

\begin{lemma}
\label{le2.9}Assume $u_{0}\in H^{s},\ s>3/2$ and $y_{0}\geq0.$ If $k_{1}%
\geq1,$ then the corresponding solution $u$ of (\ref{1.3}) with the initial
data $u(0)=u_{0}$ satisfies
\[
(k_{1}\pm\partial_{x})u(t,x)\geq0,\ \ \forall(t,x)\in\mathbb{R_{+}}%
\times\mathbb{R}.
\]

\end{lemma}

\begin{proof}
By Lemma \ref{le2.1} and a simple density argument, it suffices to show the
lemma for $s=3.$ In view of Lemma \ref{le2.5}, the potential
$y(t,x)=(1-\partial_{x}^{2})u\geq0,\ \forall(t,x)\in\mathbb{R_{+}}%
\times\mathbb{R}.$ Note $u=(1-\partial_{x}^{2})^{-1}y.$ Then we have
\begin{equation}
u(t,x)=\frac{e^{-x}}{2}\int_{-\infty}^{x}e^{\eta}y(t,\eta)d\eta+\frac{e^{x}%
}{2}\int_{x}^{\infty}e^{-\eta}y(t,\eta)d\eta\label{2.3}%
\end{equation}
and
\begin{equation}
u_{x}(t,x)=-\frac{e^{-x}}{2}\int_{-\infty}^{x}e^{\eta}y(t,\eta)d\eta
+\frac{e^{x}}{2}\int_{x}^{\infty}e^{-\eta}y(t,\eta)d\eta. \label{2.4}%
\end{equation}
It then follows from the above two relations (\ref{2.3}) and (\ref{2.4}) that
\begin{equation}
(k_{1}\pm\partial_{x})u=\frac{1}{2}(k_{1}\mp1)e^{-x}\int_{-\infty}^{x}e^{\eta
}yd\eta+\frac{1}{2}(k_{1}\pm1)e^{x}\int_{x}^{\infty}e^{-\eta}yd\eta\ \geq0.
\label{2.5}%
\end{equation}

\end{proof}

\begin{lemma}
\label{le2.10} Let $w (t, x) = (k_{1} \pm\partial_{x}) u (t, x). $ Assume
$u_{0} \in H^{s}, \ s > 3/2 $ and $y_{0} \ge0 .$ If $k_{1} \ge1 $ and $k_{2}
\ge2, $ then we have
\[
(k_{2} \pm\partial_{x}) (4 - \partial_{x}^{2})^{-1} w \ge0.
\]

\end{lemma}

\begin{proof}
In view of Lemma \ref{le2.9}, we have $w(t,x)\geq0,\forall(t,x)\in
\mathbb{R_{+}}\times\mathbb{R}.$ A simple calculation shows
\begin{align*}
(4-\partial_{x}^{2})^{-1}w  &  =\frac{1}{4}\int_{-\infty}^{\infty}%
e^{-2|x-\xi|}w(t,\xi)d\xi\\
&  =\frac{1}{4}\ e^{-2x}\int_{-\infty}^{x}e^{2\xi}w(t,\xi)d\xi+\frac{1}%
{4}\ e^{2x}\int_{x}^{\infty}e^{-2\xi}w(t,\xi)d\xi
\end{align*}
and
\[
\partial_{x}\left(  (4-\partial_{x}^{2})^{-1}w\right)  =-\frac{1}{2}%
\ e^{-2x}\int_{-\infty}^{x}e^{2\xi}w(t,\xi)d\xi+\frac{1}{2}\ e^{2x}\int
_{x}^{\infty}e^{-2\xi}w(t,\xi)d\xi.
\]
Combining above two identities, we get
\[%
\begin{split}
(k_{2}\pm\partial_{x})(4-\partial_{x}^{2})^{-1}w  &  =\frac{1}{4}\ (k_{2}%
\mp2)e^{-2x}\int_{-\infty}^{x}e^{2\xi}w(t,\xi)d\xi\\
&  +\frac{1}{4}\ (k_{2}\pm2)e^{2x}\int_{x}^{\infty}e^{-2\xi}w(t,\xi)d\xi
\ \geq0.
\end{split}
\]

\end{proof}

\section{Proof of stability}

In this primary section of the paper, we prove the stability theorem (Theorem
\ref{th1}) stated in the introduction. Note that the assumptions on the
initial profiles guarantee the existence of an unique global solution of
equation (\ref{1.3}). The stability theorem provides a quantitative estimate
of how closely the wave must approximate the peakon initially in order to be
close enough to some translate of the peakon at any later time. That translate
must be located at a point where the wave is tallest. The proof of Theorem
\ref{th1} is based on a series of lemmas including some in the previous section.

We take the wave speed $c=1$ and the case of general $c$ follows by scaling
the estimates.

Note that $\varphi(x)=e^{-|x|}\in H^{1}(\mathbb{R})$ has the peak at $x=0$
and
\begin{equation}
E_{3}(\varphi)=\int_{-\infty}^{\infty}e^{-3|x|}dx=\frac{2}{3}.\label{3.1}%
\end{equation}
Define $v_{u}=(4-\partial_{x}^{2})^{-1}u=\frac{1}{4}e^{-2|x|}\ast u.$ Then
\begin{equation}
v_{\varphi}(x)=\frac{1}{4}\int_{\mathbb{R}}e^{-2|\eta-x|}e^{-|\eta|}%
d\eta=\frac{1}{3}e^{-|x|}-\frac{1}{6}e^{-2|x|},\ \ \ \forall x\in
\mathbb{R},\label{3.2}%
\end{equation}
and thus
\begin{equation}
\max_{x\in\mathbb{R}}v_{\varphi}=v_{\varphi}(0)=\frac{1}{6}.\label{3.3}%
\end{equation}
Note $\varphi-\partial_{x}^{2}\varphi=2\delta.$ Here, $\delta$ denotes the
Dirac distribution. For simplicity, we abuse notation by writing integrals
instead of the $H^{-1}/H^{1}$ duality pairing. Hence we have
\begin{equation}%
\begin{split}
E_{2}(\varphi) &  =\Vert\varphi\Vert_{X}^{2}=\int_{\mathbb{R}}(1-\partial
_{x}^{2})\varphi\ (4-\partial_{x}^{2})^{-1}\varphi\ dx\\
&  =2\int_{\mathbb{R}}\delta(x)(4-\partial_{x}^{2})^{-1}\varphi
(x)\ dx=2v_{\varphi}(0)=\frac{1}{3}.
\end{split}
\label{3.4}%
\end{equation}

\begin{lemma}
\label{le3.1} For any $u \in L^{2}(\mathbb{R}) $ and $\xi\in\mathbb{R}, $ we
have
\[
E_{2}(u) - E_{2}(\varphi) = \| u - \varphi(\cdot- \xi) \|^{2}_{X} + 4 \left(
v_{u}(\xi) - v_{\varphi}(0) \right)  ,
\]
where $v_{u} = (4 - \partial_{x}^{2})^{-1} u. $
\end{lemma}

\begin{proof}
This can be done by a simple calculation. To see this, we have
\[%
\begin{split}
\Vert u-\varphi(\cdot-\xi)\Vert_{X}^{2}  &  =\Vert u\Vert_{X}^{2}+\Vert
\varphi\Vert_{X}^{2}-2\int_{\mathbb{R}}(1-\partial_{x}^{2})\varphi
(x-\xi)(4-\partial_{x}^{2})^{-1}u(x)dx\\
&  =\Vert u\Vert_{X}^{2}+\Vert\varphi\Vert_{X}^{2}-4\int_{\mathbb{R}}%
\delta(x-\xi)(4-\partial_{x}^{2}){-1}u(x)dx\\
&  =\Vert u\Vert_{X}^{2}+\Vert\varphi\Vert_{X}^{2}-4v_{u}(\xi)=\Vert
u\Vert_{X}^{2}-\Vert\varphi\Vert_{X}^{2}+2\Vert\varphi\Vert_{X}^{2}-4v_{u}%
(\xi)\\
&  =E_{2}(u)-E_{2}(\varphi)+4\left(  v_{\varphi}(0)-v_{u}(\xi)\right)  .
\end{split}
\]
where use has been made of integration by parts and the fact that
$E_{2}(\varphi)=\Vert\varphi\Vert_{X}^{2}=2v_{\varphi}(0).$ This completes the
proof of the lemma.
\end{proof}

In the next two lemmas, we establish two formulas related the critical values
of $v_{u}$ to the two invariants $E_{2}(u)\ $and $E_{3}(u)$. Consider a
function $0\neq u\in H^{s},s>3/2$ and $u\geq0$. Then $0<v_{u}=\left(
4-\partial_{x}^{2}\right)  ^{-1}u\in H^{s+2}\subset C^{2}$. Since $v_{u}$ is
positive and decays at infinity, it must have $n$ points $\left\{  \xi
_{i}\right\}  _{i=1}^{n}\ $ with local maximal values and $n-1$ points
$\left\{  \eta_{i}\right\}  _{i=1}^{n-1}$ with local minimal values for some
integer $n\geq1$. We arrange these critical points in their order by
\[
-\infty<\xi_{1}<\eta_{1}<\xi_{2}<\eta_{2}<...<\xi_{i-1}<\eta_{i-1}<\xi
_{i}<\eta_{i}<...<\eta_{n-1}<\xi_{n}<+\infty.
\]
Let
\begin{equation}
v_{u}(\xi_{i})=M_{i},\ 1\leq i\leq n\ \ \ \ \ \text{and\ \ \ }v_{u}(\eta
_{i})=m_{i},\ 1\leq i\leq n-1. \label{notation-M-n}%
\end{equation}
Here, we assume $n<+\infty$, that is, there are a finite number of minima and
maxima of $v_{u}$. In the case when there are infinitely many maxima and
minima, the proofs below can be modified simply by changing the finite sums to
infinite sums.

\begin{lemma}
\label{le3.2}Let $0\neq u\in H^{s},s>3/2$ and $u\geq0$. By the above
notations, define the function $g$ by
\begin{equation}
g(x)=\left\{
\begin{array}
[c]{ll}%
2v_{u}+\partial_{x}^{2}v_{u}-3\partial_{x}v_{u},\, & \eta_{i-1}<x<\xi_{i},\\
2v_{u}+\partial_{x}^{2}v_{u}+3\partial_{x}v_{u},\, & \xi_{i}<x<\eta_{i},
\end{array}
\right.  \ \ \ 1\leq i\leq n. \label{3.6}%
\end{equation}
with $\eta_{0}=-\infty$ and $\eta_{n}=+\infty.$Then we have
\begin{equation}
\int_{\mathbb{R}}g^{2}(x)dx=E_{2}(u)-12\left(  \sum_{i=1}^{n}M_{i}^{2}%
-\sum_{i=1}^{n-1}m_{i}^{2}\right)  . \label{3.7}%
\end{equation}

\end{lemma}

\begin{proof}
To simplify notations, we use $v$ for $v_{u}$ below. Then $u=4v-\partial
_{x}^{2}v$. First, we note that
\begin{align*}
E_{2}(u)  &  =\int_{\mathbb{R}}\left[  (1-\partial_{x}^{2})u\right]
vdx=\int_{\mathbb{R}}\left(  uv+\partial_{x}u\partial_{x}v\right)  dx\\
&  =\int\left\{  \left(  4v-\partial_{x}^{2}v\right)  v+\left[  \left(
4-\partial_{x}^{2}\right)  \partial_{x}v\right]  \partial_{x}v\right\}  dx\\
&  =\int\left[  4v^{2}+5\left(  \partial_{x}v\right)  ^{2}+\left(
\partial_{x}^{2}v\right)  ^{2}\right]  dx.
\end{align*}
To show (\ref{3.7}), we evaluate the integral of $g^{2}$ on each interval
$\left[  \eta_{i-1},\eta_{i}\right]  $, $1\leq i\leq n$. We have
\begin{align*}
\int_{\eta_{i-1}}^{\eta_{i}}g^{2}(x)dx  &  =\int_{\eta_{i-1}}^{\xi_{i}}\left(
2v+\partial_{xx}v-3\partial_{x}v\right)  ^{2}dx+\int_{\xi_{i}}^{\eta_{i}%
}\left(  2v+\partial_{xx}v+3\partial_{x}v\right)  ^{2}dx\\
&  =I+II.
\end{align*}
To estimate the first term, by integration by parts, we have
\begin{align*}
I  &  =\int_{\eta_{i-1}}^{\xi_{i}}\left(  4v^{2}+\left(  \partial
_{xx}v\right)  ^{2}+9\left(  \partial_{x}v\right)  ^{2}+4v\partial
_{xx}v-12v\partial_{x}v-6\partial_{xx}v\partial_{x}v\right)  dx\\
&  =\int_{\eta_{i-1}}^{\xi_{i}}\left(  4v^{2}+\left(  \partial_{xx}v\right)
^{2}+5\left(  \partial_{x}v\right)  ^{2}\right)  dx-6\left(  v\left(  \xi
_{i}\right)  ^{2}-v\left(  \eta_{i-1}\right)  ^{2}\right)  ,
\end{align*}
where use has been made of the fact that $\partial_{x}v\left(  \xi_{i}\right)
=\partial_{x}v\left(  \eta_{i-1}\right)  =0$. Similarly,
\[
II=\int_{\xi_{i}}^{\eta_{i}}\left(  4v^{2}+\left(  \partial_{x}^{2}v\right)
^{2}+5\left(  \partial_{x}v\right)  ^{2}\right)  dx+6\left(  v\left(  \eta
_{i}\right)  ^{2}-v\left(  \xi_{i}\right)  ^{2}\right)  .
\]
So
\[
\int_{\eta_{i-1}}^{\eta_{i}}g^{2}(x)dx=\int_{\eta_{i-1}}^{\eta_{i}}\left(
4v^{2}+\left(  \partial_{x}^{2}v\right)  ^{2}+5\left(  \partial_{x}v\right)
^{2}\right)  dx-12v\left(  \xi_{i}\right)  ^{2}+6v\left(  \eta_{i-1}\right)
^{2}+6v\left(  \eta_{i}\right)  ^{2}%
\]
and
\begin{align*}
\int_{\mathbb{R}}g^{2}(x)dx  &  =\int_{\mathbb{R}}\left(  4v^{2}+\left(
\partial_{x}^{2}v\right)  ^{2}+5\left(  \partial_{x}v\right)  ^{2}\right)
dx-\sum_{i=1}^{n}\left(  12v\left(  \xi_{i}\right)  ^{2}-6v\left(  \eta
_{i-1}\right)  ^{2}-6v\left(  \eta_{i}\right)  ^{2}\right) \\
&  =E_{2}\left(  u\right)  -12\left(  \sum_{i=1}^{n}M_{i}^{2}-\sum_{i=1}%
^{n-1}m_{i}^{2}\right)  ,
\end{align*}
where use has been made of the fact that $v\left(  \eta_{0}\right)  =v\left(
\eta_{n}\right)  =0$ and the notations in (\ref{notation-M-n}).
\end{proof}

\begin{lemma}
\label{le3.3} With the same assumptions and notations in Lemma \ref{le3.2}.
Define the function $h$ by
\begin{equation}
h(x)=\left\{
\begin{array}
[c]{ll}%
-\partial_{x}^{2}v_{u}-6\partial_{x}v_{u}+16v_{u},\, & \eta_{i-1}<x<\xi_{i},\\
-\partial_{x}^{2}v_{u}+6\partial_{x}v_{u}+16v_{u},\, & \xi_{i}<x<\eta_{i},
\end{array}
\right.  \ \ \ 1\leq i\leq n. \label{3.11}%
\end{equation}
with $\eta_{0}=-\infty$ and $\eta_{n}=+\infty.$ Then we have
\begin{equation}
\int_{\mathbb{R}}h(x)g^{2}(x)dx=E_{3}(u)-144\left(  \sum_{i=1}^{n}M_{i}%
^{3}-\sum_{i=1}^{n-1}m_{i}^{3}\right)  . \label{3.12}%
\end{equation}

\end{lemma}

\begin{proof}
We still use $v$ for $v_{u}$. First, note that
\[
E_{3}(u)=\int_{\mathbb{R}}\left(  4v-\partial_{xx}v\right)  ^{3}%
dx=\int_{\mathbb{R}}\left[  -\left(  \partial_{xx}v\right)  ^{3}+12\left(
\partial_{xx}v\right)  ^{2}v-48v^{2}\partial_{xx}v+64v^{3}\right]  dx.
\]

To show (\ref{3.12}), we evaluate the integral of $h(x)g^{2}(x)$ on each
interval $\left[  \eta_{i-1},\eta_{i}\right]  $, $1\leq i\leq n$. We have
\begin{align*}
\int_{\eta_{i-1}}^{\eta_{i}}h\left(  x\right)  g^{2}(x)dx  &  =\int
_{\eta_{i-1}}^{\xi_{i}}\left(  -\partial_{xx}v-6\partial_{x}v+16v\right)
\left(  2v+\partial_{xx}v-3\partial_{x}v\right)  ^{2}dx\\
&  +\int_{\xi_{i}}^{\eta_{i}}\left(  -\partial_{xx}v+6\partial_{x}%
v+16v\right)  \left(  2v+\partial_{xx}v+3\partial_{x}v\right)  ^{2}dx\\
&  =I+II.
\end{align*}

It is found that the first term
\begin{align*}
I  &  =\int_{\eta_{i-1}}^{\xi_{i}}\{-\left(  \partial_{xx}v\right)
^{3}+12\left(  \partial_{xx}v\right)  ^{2}v+27\partial_{xx}v\left(
\partial_{x}v\right)  ^{2}-108v\partial_{xx}v\partial_{x}v+60v^{2}%
\partial_{xx}v\\
\ \ \ \  &  \ \ \ \ \ \ \ \ \ \ -54\allowbreak\left(  \partial_{x}v\right)
^{3}+216\left(  \partial_{x}v\right)  ^{2}v-216v^{2}\partial_{x}%
v+64v^{3}\}dx\\
&  =\int_{\eta_{i-1}}^{\xi_{i}}\{-\left(  \partial_{xx}v\right)
^{3}+12\left(  \partial_{xx}v\right)  ^{2}v+54\left(  \partial_{x}v\right)
^{3}+60v^{2}\partial_{xx}v-54\allowbreak\left(  \partial_{x}v\right)  ^{3}\\
&  \ \ \ \ \ \ \ \ \ \ -108v^{2}\partial_{xx}v+64v^{3}\}dx-72\left(  v\left(
\xi_{i}\right)  ^{3}-v\left(  \eta_{i-1}\right)  ^{3}\right) \\
&  =\int_{\eta_{i-1}}^{\xi_{i}}\left[  -\left(  \partial_{xx}v\right)
^{3}+12\left(  \partial_{xx}v\right)  ^{2}v-48v^{2}\partial_{xx}%
v+64v^{3}\right]  dx-72\left(  v\left(  \xi_{i}\right)  ^{3}-v\left(
\eta_{i-1}\right)  ^{3}\right)  ,
\end{align*}
where use has been made of the following integral identities due to
integration by parts and $\partial_{x}v\left(  \xi_{i}\right)  =\partial
_{x}v\left(  \eta_{i-1}\right)  =0,$%
\[
\int_{\eta_{i-1}}^{\xi_{i}}\partial_{xx}v\left(  \partial_{x}v\right)
^{2}dx=\frac{1}{3}\int_{\eta_{i-1}}^{\xi_{i}}\partial_{x}\left(  \left(
\partial_{x}v\right)  ^{3}\right)  dx=0
\]%
\[
\int_{\eta_{i-1}}^{\xi_{i}}v\partial_{xx}v\partial_{x}v\ dx=\int_{\eta_{i-1}%
}^{\xi_{i}}v\partial_{x}\left(  \frac{1}{2}\left(  \partial_{x}v\right)
^{2}\right)  dx=-\frac{1}{2}\int_{\eta_{i-1}}^{\xi_{i}}\left(  \partial
_{x}v\right)  ^{3}dx,
\]%
\[
\int_{\eta_{i-1}}^{\xi_{i}}\left(  \partial_{x}v\right)  ^{2}v\ dx=\int
_{\eta_{i-1}}^{\xi_{i}}\partial_{x}v\partial_{x}\left(  \frac{1}{2}%
v^{2}\right)  dx=-\frac{1}{2}\int_{\eta_{i-1}}^{\xi_{i}}v^{2}\partial
_{xx}vdx,
\]%
\[
\int_{\eta_{i-1}}^{\xi_{i}}v^{2}\partial_{x}vdx=\int_{\eta_{i-1}}^{\xi_{i}%
}\frac{1}{3}\partial_{x}\left(  v^{3}\right)  dx=\frac{1}{3}\left(  v\left(
\xi_{i}\right)  ^{3}-v\left(  \eta_{i-1}\right)  ^{3}\right)  .
\]
Similarly,
\[
II=\int_{\xi_{i}}^{\eta_{i}}\left[  -\left(  \partial_{xx}v\right)
^{3}+12\left(  \partial_{xx}v\right)  ^{2}v-48v^{2}\partial_{xx}%
v+64v^{3}\right]  dx+72\left(  v\left(  \eta_{i}\right)  ^{3}-v\left(  \xi
_{i}\right)  ^{3}\right)
\]
and thus%
\begin{align*}
\int_{\eta_{i-1}}^{\eta_{i}}h\left(  x\right)  g^{2}(x)dx  &  =\int
_{\eta_{i-1}}^{\eta_{i}}\left[  -\left(  \partial_{xx}v\right)  ^{3}+12\left(
\partial_{xx}v\right)  ^{2}v-48v^{2}\partial_{xx}v+64v^{3}\right]  dx\\
&  -144v\left(  \xi_{i}\right)  ^{3}+72\left(  v\left(  \eta_{i-1}\right)
^{3}+v\left(  \eta_{i}\right)  ^{3}\right)  .
\end{align*}
By adding up the above integral from $1$ to $n,$ we get
\begin{align*}
\int_{\mathbb{R}}h(x)g^{2}(x)dx  &  =\int_{\mathbb{R}}\left[  -\left(
\partial_{xx}v\right)  ^{3}+12\left(  \partial_{xx}v\right)  ^{2}%
v-48v^{2}\partial_{xx}v+64v^{3}\right]  dx\\
&  \ \ \ \ -144\sum_{i=1}^{n}v\left(  \xi_{i}\right)  ^{3}+72\sum_{i=1}%
^{n}\left(  v\left(  \eta_{i-1}\right)  ^{3}+v\left(  \eta_{i}\right)
^{3}\right) \\
&  =E_{3}(u)-144\left(  \sum_{i=1}^{n}M_{i}^{3}-\sum_{i=1}^{n-1}m_{i}%
^{3}\right)  .
\end{align*}

\end{proof}

Without changing the integral identities (\ref{3.7}) and (\ref{3.12}), we can
rearrange $M_{i}$ and $m_{i}$ in the order:%
\[
M_{1}\geq M_{2}\cdots\geq M_{n}\geq0,\ \ m_{1}\geq\cdots\geq m_{n-1}\geq0.
\]
Moreover, since each local minimum is less than the neighboring local maximum,
we have $M_{i}\geq m_{i-1}$ $\left(  2\leq i\leq n\right)  $. The following
two elementary inequalities are needed in the later proofs.

\begin{lemma}
\label{lemma-inequality}For any $n\geq2$, assume $\left\{  M_{i}\right\}
_{i=1}^{n}$ and $\left\{  m_{i}\right\}  _{i=1}^{n-1}$are $2n-1$ numbers
$\ $satisfy
\[
M_{1}\geq M_{2}\cdots\geq M_{n}\geq0,\ \ m_{1}\geq\cdots\geq m_{n-1}\geq0\
\]
and $M_{i}\geq m_{i-1}$ $\left(  2\leq i\leq n\right)  $. Then

(i)
\begin{equation}
\sum_{i=2}^{n}\left(  M_{i}^{3}-m_{i-1}^{3}\right)  \leq\frac{3}{2}M_{1}%
\sum_{i=2}^{n}\left(  M_{i}^{2}-m_{i-1}^{2}\right)  . \label{ineq-easy}%
\end{equation}

(ii)
\begin{equation}
\left(  M_{1}^{2}+\sum_{2}^{n}\left(  M_{i}^{2}-m_{i-1}^{2}\right)  \right)
^{\frac{1}{2}}\geq\left(  M_{1}^{3}+\sum_{2}^{n}\left(  M_{i}^{3}-m_{i-1}%
^{3}\right)  \right)  ^{\frac{1}{3}}. \label{3.16}%
\end{equation}

\end{lemma}

\begin{proof}
\ (i) \ \ For any $2\leq i\leq n$, we have
\begin{align*}
&  \left(  M_{i}^{3}-m_{i-1}^{3}\right)  -\frac{3}{2}M_{1}\left(  M_{i}%
^{2}-m_{i-1}^{2}\right) \\
&  =-\frac{1}{2}\left(  M_{i}-m_{i-1}\right)  \left(  3M_{1}m_{i-1}%
+3M_{1}M_{i}-2M_{i}m_{i-1}-2M_{i}^{2}-2m_{i-1}^{2}\right)  \leq0
\end{align*}
since $M_{1}\geq M_{i}\geq m_{i-1}.$ This implies that desired result in
(\ref{ineq-easy}).

\vskip0.1cm \noindent(ii) \ \ Denote
\begin{equation}
A_{n}=\left(  M_{1}^{2}+\sum_{2}^{n}\left(  M_{i}^{2}-m_{i-1}^{2}\right)
\right)  ^{\frac{1}{2}}\ \text{and}\ \ B_{n}=\left(  M_{1}^{3}+\sum_{2}%
^{n}\left(  M_{i}^{3}-m_{i-1}^{3}\right)  \right)  ^{\frac{1}{3}}.
\label{3.17}%
\end{equation}
We want to show $A_{n}\geq B_{n}$ $\left(  n\geq2\right)  $ by induction. For
the case of $n=2$, it is equivalent to show that
\[
\left(  M_{1}^{2}+M_{2}^{2}-m_{1}^{2}\right)  ^{3}-\left(  M_{1}^{3}+M_{2}%
^{3}-m_{1}^{3}\right)  ^{2}\geq0
\]
if $M_{1}\geq M_{2}\geq m_{1}\geq0$. We have
\begin{align*}
&  \left(  M_{1}^{2}+M_{2}^{2}-m_{1}^{2}\right)  ^{3}-\left(  M_{1}^{3}%
+M_{2}^{3}-m_{1}^{3}\right)  ^{2}\\
&  =\allowbreak3M_{1}^{4}M_{2}^{2}-3M_{1}^{4}m_{1}^{2}-2M_{1}^{3}M_{2}%
^{3}+2M_{1}^{3}m_{1}^{3}+\allowbreak3M_{1}^{2}M_{2}^{4}\\
&  -6M_{1}^{2}M_{2}^{2}m_{1}^{2}+3M_{1}^{2}m_{1}^{4}-3M_{2}^{4}\allowbreak
m_{1}^{2}+2M_{2}^{3}m_{1}^{3}+3M_{2}^{2}m_{1}^{4}-2m_{1}^{6}\\
&  =\left(  M_{2}-m_{1}\right)  \left(  M_{1}-m_{1}\right)  {\large \{}%
3M_{1}M_{2}^{3}+3M_{1}^{3}M_{2}-2M_{1}m_{1}^{3}+3M_{1}^{3}m_{1}\\
&  -2M_{2}m_{1}^{3}+3M_{2}^{3}m_{1}-2m_{1}^{4}-2M_{1}^{2}M_{2}^{2}+M_{1}%
^{2}m_{1}^{2}+M_{2}^{2}m_{1}^{2}\\
&  -2M_{1}M_{2}m_{1}^{2}+M_{1}M_{2}^{2}m_{1}+M_{1}^{2}M_{2}m_{1}%
\allowbreak{\large \},}%
\end{align*}
which is obviously nonnegative by the assumption $M_{1}\geq M_{2}\geq
m_{1}\geq0$. Assume the inequality $A_{n}\geq B_{n}$ is true for $n\leq k$
$\left(  k\geq2\right)  $. Our goal is to deduce $A_{k+1}\geq B_{k+1}$. Since
$A_{k}\geq M_{1}\geq M_{k+1}\geq m_{k},$ we have
\begin{align*}
A_{k+1}^{6}  &  =\left[  M_{1}^{2}+\sum_{2}^{k+1}\left(  M_{i}^{2}-m_{i-1}%
^{2}\right)  \right]  ^{3}\\
&  =\left(  A_{k}^{2}+M_{k+1}^{2}-m_{k}^{2}\right)  ^{3}\geq\left(  A_{k}%
^{3}+M_{k+1}^{3}-m_{k}^{3}\right)  ^{2}\text{(by the }n=2\text{ inequality)}\\
\text{ }  &  \geq\left(  B_{k}^{3}+M_{k+1}^{3}-m_{k}^{3}\right)  ^{2}\text{
(since }A_{k}\geq B_{k}\text{ by the induction assumption)}\\
&  =B_{k+1}^{6}.
\end{align*}
Thus $A_{k+1}\geq B_{k+1}$ and $A_{n}\geq B_{n}$ is true for any $n\geq2$.
\end{proof}

The following lemma is crucial in the proof of stability of the peakons.

\begin{lemma}
Assume $u_{0}\in H^{s},s>3/2$ and $y_{0}\geq0.$ Let $M_{1}=v_{u}(t,\xi
_{1})=\max_{x\in\mathbb{R}}\{v(t,x)\}$ Then for $t\geq0,$
\begin{equation}
E_{3}(u)-144B_{n}^{3}\leq18M_{1}\left(  E_{2}(u)-12A_{n}^{2}\right)
\label{3.18}%
\end{equation}
where $u$ is the global solution of equation (\ref{1.3}) with initial value
$u_{0},$ $v_{u}=(4-\partial_{x}^{2})^{-1}u,$ and $A_{n} $ and $B_{n}$ are
defined in (\ref{3.17}).
\end{lemma}

\begin{proof}
First, by Lemma \ref{le2.8} the global solution $u$ of equation (\ref{1.3})
satisfies $u(t,x)\geq0$ and $y(t,x)=u-\partial_{x}^{2}u\geq0$ for all
$(t,x)\in\mathbb{R_{+}}\times\mathbb{R}.$ We now claim that $h\leq18v$ for
$(t,x)\in\mathbb{R_{+}}\times\mathbb{R}.$ To see this, we rewrite the
expression of $h$ as
\[
h(x)=\left\{
\begin{array}
[c]{ll}%
-\left(  \partial_{x}^{2}+3\partial_{x}+12\right)  v-3\partial_{x}v+18v,\, &
\eta_{i-1}<x<\xi_{i},\\
-\left(  \partial_{x}^{2}-3\partial_{x}+2\right)  v+3\partial_{x}v+18v,\, &
\xi_{i}<x<\eta_{i},
\end{array}
\right.  \ \ \ 1\leq i\leq n.
\]
If $\eta_{i-1}<x<\xi_{i},\ 1\leq i\leq n,$ then $v_{x}>0.$ On the other hand,
it follows from Lemma \ref{le2.10} that
\[
-\left(  \partial_{x}^{2}+3\partial_{x}+2\right)  v=-(2+\partial
_{x})(4-\partial_{x}^{2})^{-1}(1+\partial_{x})u\leq0.
\]
Hence
\begin{equation}
-\left(  \partial_{x}^{2}+3\partial_{x}+2\right)  v-3\partial_{x}%
v+18v\leq18v.\label{3.20}%
\end{equation}
A similar argument also shows that for $\xi_{i}<x<\eta_{i},\ 1\leq i\leq n,$
$\partial_{x}v<0$ and
\begin{align}
&  -\left(  \partial_{x}^{2}-3\partial_{x}+2\right)  v+3\partial
_{x}v+18v\label{3.21'}\\
&  =-(2-\partial_{x})(4-\partial_{x}^{2})^{-1}(1-\partial_{x})u+3\partial
_{x}v+18v\leq18v.\nonumber
\end{align}
The combination of (\ref{3.20}) and (\ref{3.21'}) yields
\[
h\leq18v\leq18\max v=18M_{1},\ \ \ \forall\ (t,x)\in\mathbb{R_{+}}%
\times\mathbb{R}.
\]
By the notation in (\ref{3.17}), the integral identities (\ref{3.7}) and
(\ref{3.12}) become
\[
\int_{\mathbb{R}}g^{2}(x)dx=E_{2}(u)-12A_{n}^{2}%
\]
and
\[
\int_{\mathbb{R}}h(x)g^{2}(x)dx=E_{3}(u)-144B_{n}^{3}.
\]
Note that when $n=1,\ A_{1}=B_{1}=M_{1}$. Relating the above integrals, we
get
\[
E_{3}(u)-144B_{n}^{3}\leq18M_{1}(E_{2}(u)-12A_{n}^{2}).
\]

\end{proof}

\vskip 0.1cm

\begin{lemma}
\label{le3.7} Assume $u\in H^{s},\ s>3/2$ and $y\geq0.$ Let
\[
M_{1}=v_{u}(t,\xi_{1})=\max_{x\in\mathbb{R}}\{v(t,x)\}
\]
and
\[
A_{n}=M_{1}^{2}+\sum_{i = 2}^{n}\left(  M_{i}^{2}-m_{i-1}^{2}\right)
\]
where $M_{i}$ and $m_{i-1}\left(  2\leq i\leq n\right)  $ are local maxima and
minima of $v_{u}$. If $|E_{2}(u)-E_{2}(\varphi)|\leq\delta\ $and
$|E_{3}(u)-E_{3}(\varphi)|\leq\delta$ with $0 < \delta<1,$ then \vskip 0.1cm
(i)
\begin{equation}
\left\vert M_{1}-\frac{1}{6}\right\vert <\sqrt{\delta}, \label{3.23}%
\end{equation}
recalling that $v_{\varphi}(0)=\frac{1}{6}=\max v_{\varphi}$,

\vskip 0.1cm

(ii)
\begin{equation}
\left\vert A_{n}-\frac{1}{6}\right\vert <\sqrt{\delta}, \label{3.23'}%
\end{equation}
and \vskip 0.1cm (iii)
\begin{equation}
\sum_{i = 2}^{n}\left(  M_{i}^{2}-m_{i-1}^{2}\right)  <\frac{4}{3}\sqrt
{\delta}. \label{3.23''}%
\end{equation}

\end{lemma}

\begin{proof}
To obtain (i), we first claim that
\begin{equation}
M_{1}^{3}-\frac{1}{4}E_{2}(u)M_{1}+\frac{1}{72}E_{3}(u)\leq0. \label{ineq-M1}%
\end{equation}
\ In the case when $M_{1}$ is the only local maximum of $v_{u}$, we have
$n=1$, $A_{1}=B_{1}=M_{1}$ and (\ref{ineq-M1}) follows directly from
(\ref{3.18}). When $n\geq2$, in view of (\ref{3.18}) and inequality
(\ref{ineq-easy}) in Lemma \ref{lemma-inequality} (i), there appears the
relation
\begin{align*}
&  M_{1}^{3}-\frac{1}{4}E_{2}(u)M_{1}+\frac{1}{72}E_{3}(u)\\
&  \leq144\left(  \sum_{i=2}^{n}\left(  M_{i}^{3}-m_{i-1}^{3}\right)
-\frac{3}{2}M_{1}\sum_{i=2}^{n}\left(  M_{i}^{2}-m_{i-1}^{2}\right)  \right)
\leq0.
\end{align*}
Define the cubic polynomial $P$ by
\begin{equation}
P(y)=y^{3}-\frac{1}{4}E_{2}(u)y+\frac{1}{72}E_{3}(u). \label{poly-P}%
\end{equation}
For the peakon solution, $E_{2}(\varphi)=\frac{1}{3}$ and $E_{3}%
(\varphi)=\frac{2}{3},$ the above polynomial becomes
\begin{equation}
P_{0}(y)=y^{3}-\frac{1}{12}y+\frac{1}{108}=\left(  y-\frac{1}{6}\right)
^{2}\left(  y+\frac{1}{3}\right)  . \label{3.26}%
\end{equation}
Since
\[
P_{0}(M_{1})=P(M_{1})+\frac{1}{4}\left(  E_{2}(u)-E_{2}(\varphi)\right)
M_{1}-\frac{1}{72}\left(  E_{3}(u)-E_{3}(\varphi)\right)  ,
\]
and $P(M_{1})\leq0$ by (\ref{ineq-M1}) , it follows that
\begin{equation}
\left(  M_{1}-\frac{1}{6}\right)  ^{2}\leq\frac{3}{4}\left(  E_{2}%
(u)-E_{2}(\varphi)\right)  M_{1}-\frac{1}{24}\left(  E_{3}(u)-E_{3}%
(\varphi)\right)  . \label{inequa-M_1-E23}%
\end{equation}
On the other hand, observing $E_{2}(u)-12A_{n}^{2}\geq0$, we have%
\begin{equation}
0<M_{1}\leq A_{n}\leq\sqrt{E_{2}(u)/12}\leq\sqrt{\left(  1/3+\delta\right)
/12}<\frac{1}{3} \label{upper-bound}%
\end{equation}
when $\delta<1$. It is then inferred from (\ref{inequa-M_1-E23}) that
\[
\left\vert M_{1}-\frac{1}{6}\right\vert \leq\sqrt{\frac{1}{4}\left\vert
E_{2}(u)-E_{2}(\varphi)\right\vert +\frac{1}{24}\left\vert E_{3}%
(u)-E_{3}(\varphi)\right\vert }<\sqrt{\delta}.
\]

We now prove claim (ii). When $n=1$, $A_{1}=M_{1}$ and it is reduced to (i).
When $n\geq2$, it is thereby inferred from (\ref{3.18}) that
\begin{equation}
A_{n}^{3}-\frac{1}{4}E_{2}(u)A_{n}+\frac{1}{72}E_{3}(u)\leq0, \label{3.25}%
\end{equation}
due to $0\leq M_{1}\leq A_{n}$ and $0\leq B_{n}\leq A_{n}$ by Lemma
\ref{lemma-inequality} (ii). In consequence, (\ref{3.23'}) follows from the
same argument as in part (i).

(iii) can be obtained from (i) and (ii). In fact, combining (i) and (ii), we
have
\[
2\sqrt{\delta}\geq A_{n}-M_{1}=\frac{\sum_{2}^{n}\left(  M_{i}^{2}-m_{i-1}%
^{2}\right)  }{A_{n}+M_{1}},
\]
which implies (\ref{3.23''}) by using (\ref{upper-bound}).
\end{proof}

\vskip 0.2cm

\begin{proof}
[\sl{Proof of Theorem \ref{th1}}]Let $u\in C([0,\infty),H^{s}),\
s>3/2$ be the solution of (\ref{1.3}) with initial data
$u(0)=u_{0}.$ Since $E_{2}$ and $E_{3}$ are both conserved by the
evolution equation (\ref{1.3}), we have
\begin{equation}
E_{2}(u(t,\cdot))=E_{2}(u_{0})\ \ \ \ \mathrm{and}\ \ \ \ E_{3}(u(t,\cdot
)=E_{3}(u_{0}),\ \ \forall t\geq0. \label{3.27}%
\end{equation}
Since $\Vert u_{0}-\varphi\Vert_{X}<\varepsilon,$ we obtain
\begin{align*}
|E_{2}(u_{0})-E_{2}(\varphi)|  &  =|(\Vert u_{0}\Vert_{X}-\Vert\varphi
\Vert_{X})(\Vert u_{0}\Vert_{X}+\Vert\varphi\Vert_{X})|\\
&  \leq\varepsilon\left(  2\Vert\varphi\Vert_{X}+\varepsilon\right)
=\varepsilon\left(  \frac{2}{\sqrt{3}}+\varepsilon\right)  <2\varepsilon,
\end{align*}
under the assumption $\varepsilon<\frac{1}{2}.$ In view of (\ref{3.27}), the
assumptions of Lemma \ref{le3.7} are satisfied for $u(t,\cdot)$ and
$\delta=2\varepsilon.$ It is then inferred that
\begin{equation}
\left\vert v_{u}(t,\xi_{1}(t))-\frac{1}{6}\right\vert \leq\sqrt{2\varepsilon
},\ \ \ \forall t\geq0. \label{3.28}%
\end{equation}
By (\ref{3.27}) and Lemma \ref{le3.1}, we have
\[
\Vert u(t,\cdot)-\varphi(\cdot-\xi_{1}(t))\Vert_{X}^{2}=E_{2}(u_{0}%
)-E_{2}(\varphi)+4(v_{\varphi}(0)-v_{u}(t,\xi_{1}(t))),\ \ \ \forall t\geq0.
\]
Combining the above estimates yields
\[
\Vert u(t,\cdot)-\varphi(\cdot-\xi_{1}(t))\Vert_{X}\leq\sqrt{2\varepsilon
+4\sqrt{2\varepsilon}}<3\varepsilon^{1/4}.\ \ \ \forall t\geq0.
\]
Estimates (\ref{ineq-M1-thm}) and (\ref{ineq-M2-thm}) then follow directly
from Lemma \ref{le3.7} (ii) and (iii). This completes the proof of Theorem
\ref{th1}.
\end{proof}

\begin{remark}
\label{remark-proof}We make several comments.

(1) By (\ref{upper-bound}), $M_{1}=\max v_{u}\leq\sqrt{E_{2}\left(  u\right)
/12}.$ For peakons $c\varphi$, we have $\max v_{c\varphi}=\sqrt{E_{2}\left(
c\varphi\right)  /12}=\frac{1}{6}c$. So among all waves of a fixed energy
$E_{2}$, the peakon is tallest in terms of $v_{u}$.

(2) In our proof, we use inequality (\ref{inequa-M_1-E23}) to get estimates
(\ref{3.23}) and (\ref{3.23''}) more directly, compared with the argument in
\cite{C-S} by analyzing the root structure of the polynomial $P\left(
y\right)  $. Moreover, it implies that the peakons are energy minimizers with
a fixed invariant $E_{3},$ which explains their stability. Indeed, if
$E_{3}\left(  u\right)  =E_{3}\left(  \varphi\right)  $, it follows from
(\ref{inequa-M_1-E23}) that $E_{2}(u)\geq E_{2}(\varphi)$. The same remark
also applies to the CH equation and shows that the CH-peakons are energy
minima with fixed $F_{3}$.

(3) Compared with \cite{C-S}, our construction of the integral relations
(\ref{3.7}) and (\ref{3.12}) is more delicate. It not only is required in our
current case, but also provides us more information about stability via
(\ref{ineq-M2-thm}). For the CH equation, even if the orbital stability is
proved by a simpler construction \cite{C-S}, our approach can also give the
additional stability information. More specifically, for the CH equation
(\ref{ch}) with $y_{0}\geq0$, by refining the integrals of \cite[Lemma 2]{C-S}
to each monotonic interval of $u$, one can obtain
\[
F_{3}\left(  u\right)  -\frac{4}{3}B_{n}^{3}\leq M_{1}\left(  F_{2}\left(
u\right)  -2A_{n}^{2}\right)  \text{,}%
\]
where $F_{2}$ and $F_{3}$ are defined in (\ref{ch-invariants}), and $A_{n} $
and $B_{n}$ in (\ref{3.17}) with $M_{i} $ and $m_{i}$ being the maxima and
minima of $u$, respectively. Hence, estimate (\ref{ineq-M2-thm}) may be
obtained by following the proof of Lemma \ref{le3.7}.
\end{remark}

\vskip 0.5cm

\noindent\textbf{Acknowledgements}{\Large \ }

\vskip 0.2cm

The authors 
thank the anonymous referee for valuable comments and suggestions.
The work of Zhiwu Lin is supported partly by the NSF grants DMS-0505460 and
DMS-0707397.


\vskip 0.2cm

\begin{thebibliography}{99}                                                                                               %


\bibitem {B-C}{\small \textsc{A. Bressan and A. Constantin,} Global
conservative solutions of the Camassa-Holm equation,  \textit{Arch. Rat. Mech. Anal.,} \textbf{183} (2007), 215--239. }

\bibitem {C-H}{\small \textsc{R. Camassa and D. Holm,} An integrable shallow
water equation with peaked solitons, \textit{Phys. Rev. Letters,} \textbf{71}
(1993), 1661--1664. }

\bibitem {Co-K}{\small \textsc{G. M. Coclite and K. H. Karlsen,} On the
well-posedness of the Degasperis-Procesi equation, \textit{J. Funct. Anal.},
\textbf{233} (2006), 60--91. }

\bibitem {C-K-R}{\small \textsc{G. M. Coclite, K. H. Karlsen and N. H.
Risebro,} Numerical schemes for computing discontinuous solutions of the
Degasperis-Procesi equation, preprint. }

\bibitem {Cf}{\small \textsc{A. Constantin,} Global existence of solutions and
breaking waves for a shallow water equation: a geometric approach,
\textit{Ann. Inst. Fourier (Grenoble),} \textbf{50} (2000), 321--362. }

\bibitem {Cf2}{\small \textsc{A. Constantin,} Finite propagation speed for the
Camassa-Holm equation, \textit{J. Math. Phys.,} \textbf{46} (2005), 023506, 4
pp. }

\bibitem {C-Ep}{\small \textsc{A. Constantin and J. Escher,} Global existence
and blow-up for a shallow water equation, \textit{Annali Sc. Norm. Sup. Pisa,}
\textbf{26} (1998), 303--328. }

\bibitem {C-E}{\small \textsc{A. Constantin and J. Escher,} Wave breaking for
nonlinear nonlocal shallow water equations, \textit{Acta Mathematica,}
\textbf{181} (1998), 229--243. }


\bibitem {C-K-K-T}{\small \textsc{A. Constantin,  T. Kappeler, B. Kolev, and P. Topalov,} On geodesic exponontial maps of the Virasoro group,  \textit{Ann. glob. anal. Geom.,} \textbf{31} (2007), 155--180. }

\bibitem {C-K}{\small \textsc{A. Constantin and B. Kolev,} Geodesic flow on
the diffeomorphism group of the circle, \textit{Comment. Math. Helv.,}
\textbf{78} (2003), 787--804. }

\bibitem {C-M2}{\small \textsc{A. Constantin and L. Molinet,} Obtital
stability of solitary waves for a shallow water equation, \textit{Physica D,}
\textbf{157} (2001), 75--89. }

\bibitem {C-S}{\small \textsc{A. Constantin and W. A. Strauss,} Stability of
peakons, \textit{Comm. Pure Appl. Math.,} \textbf{53} (2000), 603--610. }

\bibitem {D-H-H}{\small \textsc{A. Degasperis, D. D. Holm, and A. N. W. Hone,}
A New Integral Equation with Peakon Solutions, \textit{Theoretical and
Mathematical Physics,} \textbf{133} (2002), 1463--1474. }

\bibitem {D-P}{\small \textsc{A. Degasperis and M. Procesi,} Asymptotic
integrability, in Symmetry and Perturbation Theory, edited by A. Degasperis
and G. Gaeta, World Scientific (1999), 23--37.}

\bibitem {D-G-H1}{\small \textsc{H. R. Dullin, G. A. Gottwald, and D. D.
Holm,} An integrable shallow water equation with linear and nonlinear
dispersion, \textit{Phys. Rev. Letters,} \textbf{87} (2001), 4501--4504. }

\bibitem {E-L-Y}{\small \textsc{J. Escher, Y. Liu, and Z. Yin,} Global weak
solutions and blow-up structure for the Degasperis-Procesi equation,
\textit{J. Funct. Anal.}, \textbf{241} (2006), 457--485. }

\bibitem {E-L-Y2}{\small \textsc{J. Escher, Y. Liu, and Z. Yin,} Shock waves
and blow-up phenomena for the periodic Degasperis-Procesi equation,
\textit{Indiana Univ. Math. J.,} \textbf{56} (2007), 87--117. }

\bibitem {F-F}{\small \textsc{A. Fokas and B. Fuchssteiner,} Symplectic
structures, their B{\"{a}}cklund transformation and hereditary symmetries,
\textit{Physica D,} \textbf{4} (1981), 47--66. }

\bibitem {H}{\small \textsc{D. Henry,} Infinite propagation speed for the
Degasperis-Procesi equation. \textit{J. Math. Anal. Appl.,} \textbf{311}
(2005), 755--759. }

\bibitem {H-S}{\small \textsc{D. D. Holm and M. F. Staley,} Wave structure and
nonlinear balances in a family of evolutionary PDEs, \textit{SIAM J. Appl.
Dyn. Syst. (electronic),} \textbf{2} (2003), 323--380. }

\bibitem {H-S2}{\small \textsc{D. D. Holm and M. F. Staley,} Nonlinear balance
and exchange of stability in dynamics of solitons, peakons, ramps/cliffs and
leftons in a 1-1 nonlinear evolutionary PDE, \textit{Phys. Lett. A,}
\textbf{308} (2003), 437--444. }

\bibitem {J}{\small \textsc{R. S. Johnson,} Camassa-Holm, Korteweg-de Vries
and related models for water waves, \textit{J. Fluid Mech.,} \textbf{455}
(2002), 63--82. }

\bibitem {K}{\small \textsc{T. Kato,} Quasi-linear equations of evolution,
with applications to partial differential equations, in: Spectral Theory and
Differential Equations, Lecture Notes in Math., Springer Verlag, Berlin,
\textbf{448} (1975), 25--70. }

\bibitem {L-Y}{\small \textsc{Y. Liu and Z. Yin,} Global existence and blow-up
phenomena for the Degasperis-Procesi equation, \textit{Comm. Math. Phys.,}
\textbf{267} (2006), 801--820. }

\bibitem {Lu}{\small \textsc{H. Lundmark,} Formation and dynamics of shock
waves in the Degasperis-Procesi equation, \textit{J. Nonlinear
Science,} \textbf {17} (2007), 169--198. }

\bibitem {L-S}{\small \textsc{H. Lundmark and J. Szmigielski,} Multi-peakon
solutions of the Degasperis-Procesi equation, \textit{Inverse Problems,}
\textbf{19} (2003), 1241--1245. }

\bibitem {Ma}{\small \textsc{Y. Matsuno,} Multisoliton solutions of the
Degasperis-Procesi equation and their peakon limit, \textit{Inverse Problems,}
\textbf{21} (2005), 1553--1570. }

\bibitem {Mi}{\small \textsc{G. Misiolek,} A shallow water equation as a
geodesic flow on the Bott-Virasoro group, \textit{J. Geom. Phys.,} \textbf{24}
(1998), 203--208. }
 
\bibitem {Mu}{\small \textsc{O. G. Mustafa,} A note on the Degasperis-Procesi equation, \textit{J. Nonlinear Math. Phys.,} \textbf{12} (2005), 10--14.}

\bibitem {Ta}{\small \textsc{T. Tao,} Low-regularity global solutions to nonlinear dispersive equations. Surveys in analysis and operator theory (Canberra, 2001), \textit{Proc.Centre Math. Appl. Austral. Nat. Univ., Canberra, } \textbf{40} (2002), 19--48.}
    
\bibitem {Wh}{\small \textsc{G. B. Whitham,} Linear and Nonlinear Waves,
\textit{J. Wiley \& Sons, New York,} 1980. }

\bibitem {Y1}{\small \textsc{Z. Yin,} On the Cauchy problem for an integrable
equation with peakon solutions, \textit{Illinois J. Math.,} \textbf{47}
(2003), 649--666. }

\bibitem {Y4}{\small \textsc{Z. Yin,} Global solutions to a new integrable
equation with peakons, \textit{Indiana Univ. Math. J.,} \textbf{53} (2004),
1189--1210. }
\end{thebibliography}
\end{document}